\documentclass[12pt]{amsart}
\usepackage[utf8]{inputenc}
\usepackage[T1]{fontenc}
\usepackage{url}
\usepackage{color}
\usepackage{hyperref}
\usepackage{enumerate}
\usepackage{array}
\usepackage{nicefrac}
\usepackage[section]{placeins}
\usepackage{amsmath,amssymb,enumerate,amsthm}
\usepackage{latexsym}
\usepackage{setspace}
\usepackage{mathtools}
\usepackage{longtable}
\usepackage{mathdots}
\usepackage{graphicx, subfig}
\usepackage{anysize}
\usepackage[linguistics,edges]{forest}

\usepackage{enumitem}

\theoremstyle{plain}

\newtheorem{theorem}{Theorem}[section]
\newtheorem{lemma}[theorem]{Lemma}
\newtheorem{prop}[theorem]{Proposition}
\newtheorem{coro}[theorem]{Corollary}

\theoremstyle{definition}

\newtheorem{example}[theorem]{Example}

\theoremstyle{remark}
\newtheorem{remark}[theorem]{Remark}

\numberwithin{equation}{section}

\def\H{\mathbb{H}}

\def\N{\mathbb{N}}
\def\C{\mathbb{C}}
\def\R{\mathbb{R}}
\def\Z{\mathbb{Z}}

\def\M{\mathcal{M}}

\def\SL{\mathrm{SL}(2,\mathbb{Z})}
\def\GL{\mathrm{GL}(2,\mathbb{Z})}

\def\Im{\mathrm{Im}}
\def\Re{\mathrm{Re}}

\def\val{\mathrm{val}}
\def\op{\mathrm{op}}

\def\GLN{\mathrm{GL}(2,\N_0)}

\usepackage{tikz-cd}

\title{Real part of cycle integrals and conjectures of Kaneko}

\author{P.~Bengoechea}
\address{Universitat de Barcelona, Departament de Matem\`atiques i Inform\`atica, Gran Via de les Corts Catalanes, 585, L'Eixample, 08007 Barcelona, Spain}
\email{bengoechea@ub.edu}
\thanks{P.~Bengoechea's research is supported by Ram\'on y Cajal grant RYC2020-028959-I}

\author{S.~Herrero}
\address{Universidad de Santiago de Chile, Dept.~de Matem\'atica y Ciencia de la Computaci\'on, Av.~Libertador Bernardo O'Higgins 3363, Santiago, Chile} 
\email{sebastian.herrero.m@gmail.com}
\thanks{S.~Herrero's research is supported by ANID FONDECYT Regular grant 1250734}

\author{\"O.~Imamo\={g}lu}
\address{ETH, Mathematics Dept., CH-8092, Z\"urich, Switzerland}
\email{ozlem@math.ethz.ch}
\thanks{\"O.~Imamo\=glu's research  is supported by SNF grant 200021-185014}

\begin{document}

\maketitle

\vspace{-1em} 
\begin{center}
    \textit{Dedicated to Prof.~Masanobu Kaneko on his 60+4th birthday}
\end{center}

\begin{abstract}
   We prove two of Kaneko's conjectures on the ``values'' $\val(w)$ of the modular $j$ function at real quadratic irrationalities: we prove the lower bound $\Re(\val(w))\geq \val\left(\frac{1+\sqrt{5}}{2}\right)$ for all real quadratics $w$ and the upper bound $\Re(\val(w))\leq \val\left(1+\sqrt{2}\right)$ for all Markov irrationalities $w$. 
   These results generalize to the ``values'' at quadratic irrationalities of any weakly holomorphic modular function $f$ such that $f(e^{it})$ is real, non-negative and increasing for $t\in [\pi/3,\pi/2]$.
\end{abstract}

\section{Introduction}

Let $j$ denote Klein's modular function, i.e.~the unique holomorphic function defined on the upper-half plane $\H:=\{\tau \in \C:\Im(\tau)>0\}$ that is invariant under the action of the modular group $\SL$ by fractional linear transformations and has a Fourier expansion of the form
$$j(\tau)=\frac{1}{q}+744+\sum_{n=1}^{\infty}c_nq^n \quad \text{ with }c_n\in \C \text{ where } q:=e^{2\pi i \tau}.$$

Kaneko defined in \cite{Kan09} a function $\val(w)$ for the ``values'' of $j$ on real quadratic irrationalities $w$ in terms of cycle integrals normalized by hyperbolic length (see Section \ref{sec:cycle_integrals} for details). The real part of Kaneko's $\val$ function is $\GL$-invariant, and he observed several Diophantine properties of this function. Among others, he conjectured:
\begin{enumerate}
\item[(I)] For all real quadratic irrationalities $w$, $\Re(\val(w))\in [\val(\phi),744]$ where $\phi=\frac{1+\sqrt{5}}{2}$ is the golden ratio and $\val(\phi)=706.3248\ldots$.
\item[(II)] When restricted to Markov irrationalities $w$, the real part of the function  $\val$ satisfies  $\Re(\val(w))\in [\val(\phi),\val(\psi)]$ where $\psi=1+\sqrt{2}$ is the silver ratio and $\val(\psi)=709.8928\ldots$.
\end{enumerate}
Some of Kaneko's conjectures were proved in \cite{BI19,BI20}. In particular, the upper bound in (I) was proven in \cite[Theorem 1]{BI20}. In contrast, to the best of our knowledge no proof of the optimal lower bound in (I) or of the optimal upper bound in (II) has been announced, although weaker bounds for $\Re(\val(w))$ were obtained in \cite{BI20,Ben22,BHI25}.

The purpose of this article is to prove the lower bound in (I) and the upper bound in (II) hence completing the proof of both conjectures. Our first main theorem is the following.

\begin{theorem}\label{thm:Kaneko_lower_bound}
   For every quadratic irrationality $w\in \R$ we have $\Re(\val(w))\geq \val(\phi)$. 
\end{theorem}

We now recall the construction of Markov irrationalities. First, given a sequence of integers $(a_i)_{i=1}^\infty$ with $a_i\geq 1$ for all $i\geq 2$ we write
\[[a_1;a_2,a_3,\ldots ]:=a_1+\dfrac{1}{a_2+\dfrac{1}{a_3+\dfrac{1}{\ddots}}}\]
for the corresponding infinite continued fraction. Starting with the purely periodic continued fraction expansions
\[\phi=\frac{1+\sqrt{5}}{2}=[\overline{1;1}] \quad \text{and}\quad \psi=1+\sqrt{2}=[\overline{2;2}],\]
and using the conjunction of purely periodic continued fraction expansions
\[[\overline{a_1;\ldots ,a_r}]\odot [\overline{b_1;\ldots ,b_s}]:=[\overline{a_1;\ldots,a_r,b_1,\ldots ,b_s}]\]
in the form $w_1,w_2\mapsto w_2\odot w_1$ we obtain the Markov tree as illustrated in Figure \ref{fig:Markov_tree}. Here we use the abbreviations $[\overline{1;1}]=[\overline{1_2}]$, $[\overline{2;2}]=[\overline{2_2}]$, $[\overline{2;2,1,1}]=[\overline{2_2,1_2}]$, etc.

\begin{figure}[h!]
\centering
\begin{forest}
[,phantom [{$[\overline{1_2}]$},name=p1] [] [] [] [] [] [] 
[
[{$[\overline{2_2,1_2}]$},name=p2, no edge,tikz={\draw (p2.north)--(p1.south);}
[
{$[\overline{2_2,1_4}]$}
[
{$[\overline{2_2,1_4}]$}
[$\ldots$]
[$\ldots$]]
[
{$[\overline{2_2,1_2,2_2,1_4}]$}
[$\ldots$]
[$\ldots$]]
]
[
{$[\overline{2_4,1_2}]$}
[
{$[\overline{2_4,1_2,2_2,1_2}]$}
[$\ldots$]
[$\ldots$]] 
[
{$[\overline{2_6,1_2}]$}
[$\ldots$]
[$\ldots$]]]]]
[] [] [] [] [] []  [{$[\overline{2_2}]$},name=p3] tikz={\draw (p2.north)--(p3.south);}
]
\end{forest} \caption{The Markov tree.}\label{fig:Markov_tree}
\end{figure}
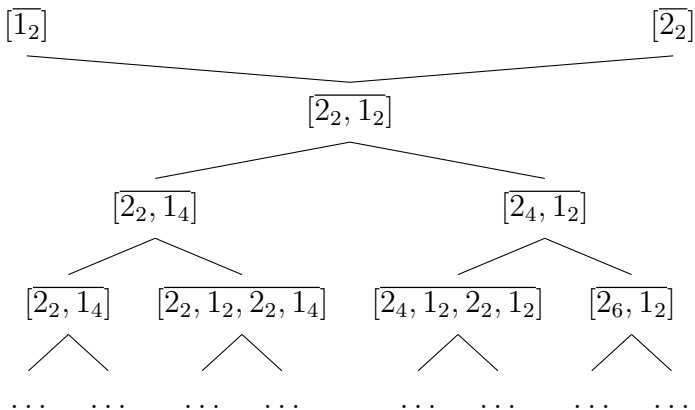

We let $\M$ denote the collection of all real quadratic irrationalities appearing in this tree. A real number $w\in \M$ is called a Markov irrationality. These real quadratics are, up to $\GL$-equivalence, the worst approximable numbers in the sense of Diophantine approximation according to Markov's famous theorem (see, e.g., \cite{Aig2013}). Our second main theorem is the following.

\begin{theorem}\label{thm:Kaneko_upper_bound_Markovs}
   For every Markov irrationality $w\in \M$ we have $\val(\phi)\leq \Re(\val(w))\leq \val(\psi)$. 
\end{theorem}

Our methods of proof for Theorems \ref{thm:Kaneko_lower_bound} and \ref{thm:Kaneko_upper_bound_Markovs} parallel the techniques and computations carried out in \cite{BHI25} where similar bounds for the values of certain \emph{Lyapunov exponent} attached to modular functions were given. The results obtained in \emph{loc.~cit.} were motivated by the conjectures of Kaneko \cite{Kan09} mentioned above and works of Spalding and Veselov \cite{SV17, SV18}. 

Theorems \ref{thm:Kaneko_lower_bound} and \ref{thm:Kaneko_upper_bound_Markovs} generalize to the ``values'' at quadratic irrationalities of any weakly holomorphic modular function $f$ such that $f(e^{it})$ is real, non-negative and increasing for $t\in [\pi/3,\pi/2]$; see Theorems \ref{thm:first_main} and \ref{thm:second_main} in Section \ref{sec:cycle_integrals}.

It is a natural question to ask if similar methods as the ones used in this paper lead to a proof of Kaneko's conjectural bounds for the imaginary part of the $\val$ function, which predict that $\Im(\val(w))\in [-1,1]$ for all quadratic irrationals $w$. We plan to return to this problem in the near future.

\subsection{Outline of the paper} This paper is organized as follows. In the next section, we recall the definition of the cycle integrals of a weakly holomorphic modular function $f$ and introduce a function $\val_f(w)$ on real quadratic irrationalities $w$ that generalizes Kaneko's $\val$ function. There we also present Theorems \ref{thm:first_main} and \ref{thm:second_main}, which are generalizations of Theorems \ref{thm:Kaneko_lower_bound} and \ref{thm:Kaneko_upper_bound_Markovs}, respectively. In Section \ref{sec:formula_Revalf} we present a formula for $\Re(\val_f(w))$ that is used in the proofs of our results. The proof strategies of Theorems \ref{thm:first_main} and \ref{thm:second_main} are presented in steps in Sections \ref{sec:strategy_first_main} and \ref{sec:strategy_second_main}, respectively. These rely on several properties of certain auxiliary functions whose proofs are given in the appendix.

\section*{Aknowledgements}

The authors thank Prof. Masanobu Kaneko for his beautiful paper \cite{Kan09} that motivated this work. We also thank the organizers of the conference ``Modular Forms and Multiple Zeta Values'' held in honor of Prof.~Kaneko's 64th birthday at Kindai University in February 2025, where some of the results of this paper were announced.

\section{Cycle integrals and Kaneko's $\val$ function}\label{sec:cycle_integrals}

Given a hyperbolic matrix $A\in \SL$ we let $\tilde{w},w\in \R$ denote the attracting and repelling fixed points of $A$, respectively. Let $Q(x,y)=ax^2+bxy+cy^2$ be an integral, primitive, indefinite quadratic form satisfying $Q(w,1)=Q(\tilde{w},1)=0$. Among the two quadratic forms $Q$ and $-Q$ we choose the one satisfying $\mathrm{sgn}(a)=\mathrm{sgn}(w-\tilde{w})$. With this convention, we have the formulas
\begin{equation*}
    w=\frac{-b+\sqrt{D}}{2a},\quad \tilde{w}=\frac{-b-\sqrt{D}}{2a},
\end{equation*}
where $D:=b^2-4ac>0$ is the discriminant of $Q$. Then, given $f:\H \to \C$ a weakly holomorphic modular function  for $\SL$ which is real and  non-negative on the geodesic arc $\{e^{it} : \pi/3\leq t\leq 2\pi/3\}$, we put
\begin{equation*}
    I_f(A):=\int_{\tau_0}^{A\tau_0}\frac{f(\tau)\sqrt{D}}{Q(\tau,1)}d\tau \quad (\text{any } \tau_0\in \H).
\end{equation*}
In the case $f\equiv 1$ one has $I_1(A)=2\log(\varepsilon)$ where $\varepsilon>1$ is the largest eigenvalue of $A$. 

Starting with a real quadratic irrationality $w$ we can define Kaneko's $\val$ function as
\[\val(w):=\frac{I_j(A)}{I_1(A)} \quad \text{where $A\in \SL$ is any hyperbolic element fixing $w$}.\]
In particular, if $w=[\overline{a_1;\ldots,a_\ell}]$ with $\ell\geq 2$ even and $a_i\geq 1$ for all $i$, then we can choose $A=T^{a_1}V^{a_2}\cdots T^{a_{\ell-1}}V^{a_\ell}$. 

In this paper, we are interested in the values of $\Re(\val(w))$. More generally, given a weakly holomorphic modular function $f$ as above, and a real quadratic irrationality $w$, we define
\begin{equation}\label{def_valf}
   \val_f(w):=\frac{I_f(A)}{I_1(A)} \quad \text{for any hyperbolic matrix $A\in \SL$ fixing $w$}. 
\end{equation}

Recall that $\phi=\frac{1+\sqrt{5}}{2}$ denotes the golden ratio and $\psi=1+\sqrt{2}$ denotes the silver ratio. Since $\phi$ is a fixed point of $\Phi:=\left(\begin{smallmatrix}
    1 & 1 \\ 1 & 0
\end{smallmatrix}\right)$ and $\psi$ is a fixed point of $\Psi:=\left(\begin{smallmatrix}
    2 & 1 \\ 1 & 0
\end{smallmatrix}\right)$, which are matrices in $\GL$ of determinant $-1$, the values $\val_f(\phi)$ and $\val_f(\psi)$ are real (see Lemma \ref{lem:prop_valf}$(ii)$ in Section \ref{sec:formula_Revalf}).

Theorems \ref{thm:Kaneko_lower_bound} and \ref{thm:Kaneko_upper_bound_Markovs} in the introduction are direct consequences of the following two results when choosing $f=j$.

\begin{theorem}\label{thm:first_main}
Let $f$ be a weakly holomorphic modular function for $\SL$ that is real and non-negative on the geodesic arc $\{e^{it} : \pi/3\leq t\leq 2\pi/3\}$. Assume that $f(e^{it})$ is increasing for $t\in [\pi/3,\pi/2]$. Then, for every quadratic irrationality $w$ we have $\Re(\val_f(w))\geq \val_f(\phi)$. 
\end{theorem}

\begin{theorem}\label{thm:second_main}
Let $f$ be as in Theorem \ref{thm:first_main} and let $w\in \M$ be a Markov irrationality. Then, we have $\val_f(\phi)\leq \Re(\val_f(w))\leq \val_f(\psi)$. 
\end{theorem}

The main ideas of the proofs of Theorems \ref{thm:first_main} and \ref{thm:first_main} are presented in Sections \ref{sec:strategy_first_main} and \ref{sec:strategy_second_main}, respectively.

\section{A formula for $\Re(\val_f(w))$}\label{sec:formula_Revalf}

Let $f$ be a weakly holomorphic modular function for $\SL$ that is real and non-negative on the geodesic arc $\{e^{it} : \pi/3\leq t\leq 2\pi/3\}$. The function $\val_f(w)$ defined in \eqref{def_valf} 
satisfies the following properties (see, e.g., \cite{Kan09}).

\begin{lemma}\label{lem:prop_valf}
For all quadratic irrationalities $w\in \R$ we have:
    \begin{enumerate}
    \item[$(i)$] $\val_f(M(w))=\val_f(w)$ for all $M\in \SL$.
    \item[$(ii)$] $\val_f(M(w))=\overline{\val_f(w)}$ for all $M\in \GL$ with $\det(M)=-1$.
\end{enumerate}
In particular, the function $\Re(\val(w))$ is $\GL$-invariant.
\end{lemma}

It is well known that every real quadratic irrationality is $\GL$-equivalent to a quadratic irrationality with purely periodic continued fraction expansion. Since $\Re(\val_f(w))$ is $\GL$-invariant, we can then assume that $w=[\overline{a_1;a_2,\ldots,a_\ell}]$ with $a_i\geq 1$ for all $i=1,\ldots,\ell$. Moreover, duplicating the period if necessary, we can assume that $\ell \geq 2$ is even.

In this section, we present a formula for $\Re(\val_f(w))$ that is at the core of our proofs of Theorems \ref{thm:first_main} and \ref{thm:second_main}. Following \cite[Section 2.5]{BHI25} we let $F:\R\times [\pi/3,2\pi/3]\to \R$ be defined by
\begin{equation*}
    F(x,t):=\frac{x}{1+x^2-2x\cos(t)},
\end{equation*}
and for a quadratic irrationality $w=[\overline{a_1;a_2,\ldots,a_\ell}]$ with $\ell\geq 2$ even and $t\in [\pi/3,2\pi/3]$ let
\begin{equation}\label{eq:def_S_F}
S_F(w,t):=\sum_{i=1}^{\ell}\sum_{m=1}^{a_i}F(w_{i,m},t) \quad \text{where }w_{i,m}:=[m;\overline{a_{i+1},\ldots,a_\ell,a_1,\ldots,a_i}].
\end{equation}
Defining $T:=\left(\begin{smallmatrix}
    1 & 1 \\ 0 & 1
\end{smallmatrix}\right)$ and $V:=\left(\begin{smallmatrix}
    1 & 0 \\ 1 & 1
\end{smallmatrix}\right)$ and letting $A=T^{a_1}V^{a_2}\cdots T^{a_{\ell-1}}V^{a_\ell}$ 
we have by \cite[Proposition 2.15]{BHI25} the formula
\begin{equation}\label{eq:real_part_cycle_integral}
\Re(I_f(A))=\int_{\pi/3}^{2\pi/3}f(e^{it})\sin(t)\big(S_F(w,t)+S_F(w^\op,t)\big)dt,
\end{equation}
where \(w^{\op}:=[\overline{a_\ell;a_{\ell-1},\ldots,a_1}]\).

A simple computation shows that
$$\int \sin(t)F(x,t)dt=\frac{1}{2}\log(1+x^2-2x\cos(t)).$$
This leads to the function $L:\R\to\R$ defined by
\begin{equation}\label{eq:def_L}
    L(x):=\int_{\pi/3}^{2\pi/3}\sin(u)F(x,u)du=\frac{1}{2}\log\left(\frac{1+x^2+x}{1+x^2-x}\right).
\end{equation}
We then define for $w=[\overline{a_1;a_2,\ldots,a_\ell}]$ the sum
\[S_L(w):=\sum_{i=1}^{\ell}\sum_{m=1}^{a_i}L(w_{i,m}) \qquad (w_{i,m}=[m;\overline{a_{i+1},\ldots, a_\ell,a_1,\ldots,a_{i}}])\]
and note that by \eqref{eq:real_part_cycle_integral} with $f\equiv 1$ we have
\(I_1(A)=S_L(w)+S_L(w^\op).\) 
Finally, let
\begin{equation}\label{eq:def_hat_S}
\hat{S}(w,t):=\frac{\sin(t)(S_F(w,t)+S_F(w^\op,t))}{I_1(A)}=\frac{\sin(t)(S_F(w,t)+S_F(w^\op,t))}{S_L(w)+S_L(w^\op)}.    
\end{equation}

The following proposition is the main result of this section.

\begin{prop}\label{prop:formula_re_valf}
Given $w=[\overline{a_1;a_2,\ldots,a_\ell}]$ with $\ell\geq 2$ even and $1\leq i\leq \ell$ put \(v_i:=[\overline{a_i;a_{i+1},\ldots,a_\ell,a_1,\ldots,a_{i-1}}]\). Then:
\begin{enumerate}
    \item[$(i)$]  \(S_L(w)=S_L(w^\op)=\sum\limits_{i=1}^\ell \log(v_i)=\log(\varepsilon)\)
    where $\varepsilon>1$ is the largest eigenvalue of $A=T^{a_1}V^{a_2}\cdots V^{a_\ell}$. In particular, $v_1\cdots v_\ell=\varepsilon$ and $I_1(A)=2\sum\limits_{i=1}^\ell \log(v_i)$.
    \item[$(ii)$] 
    \(\Re(\val_f(w))=\frac{1}{2}\int\limits_{\pi/3}^{2\pi/3}f(e^{it})\left(\hat{S}(w,t)+\hat{S}(w,\pi-t)\right)dt.\)   
\end{enumerate}
\end{prop}

\begin{remark}
\begin{enumerate}
    \item An identity analogue to $v_1\cdots v_\ell=\varepsilon$ but using negative continued fraction expansions can be found in \cite[Equation~$(6.4)$]{Zag75}.
    \item The formula for $\Re(\val_f(w))$ in Proposition \ref{prop:formula_re_valf}$(ii)$ involves an integrand that is invariant under $t\mapsto \pi-t$. This formula is convenient for our method of proof of Theorems \ref{thm:first_main} and \ref{thm:second_main}.  
\end{enumerate}
\end{remark}

\begin{proof}[Proof of Proposition \ref{prop:formula_re_valf}]
By definition of $L(x)$ (see \eqref{eq:def_L}) we have
\begin{equation*}
S_L(w)=\frac{1}{2}\sum_{i=1}^{\ell}\sum_{m=1}^{a_i}\log\left(  \frac{1+w_{i,m}+w_{i,m}^2}{1-w_{i,m}+w_{i,m}^2}\right)=\frac{1}{2}\sum_{i=1}^{\ell}\sum_{m=1}^{a_i}\log\left(  \frac{1-(w_{i,m}+1)+(w_{i,m}+1)^2}{1-w_{i,m}+w_{i,m}^2}\right).
\end{equation*}
Recognizing that the inner sum is telescopic, we arrive to
\[S_L(w)=\frac{1}{2}\sum_{i=1}^\ell\log\left(  \frac{1+[a_i;\overline{a_{i+1},\ldots,a_i}]+[a_i;\overline{a_{i+1},\ldots,a_i}]^2}{1-[1;\overline{a_{i+1},\ldots,a_i}]+[1;\overline{a_{i+1},\ldots,a_i}]^2}\right).\]
This equals
\[\frac{1}{2}\sum_{i=1}^\ell \log \left(\frac{1+v_i+v_i^2}{1-(1+1/v_{i+1})+(1+1/v_{i+1})^2}\right)=\frac{1}{2}\sum_{i=1}^\ell \log \left(\frac{1+v_i+v_i^2}{1-(1+1/v_i)+(1+1/v_i)^2}\right).\]
Using that
\(\frac{1+x+x^2}{1-(1+x^{-1})+(1+x^{-1})^2}=x^2\) 
we conclude \(S_L(w)=\sum\limits_{i=1}^\ell \log(v_i)\) as desired. We now claim that
\begin{equation}\label{product_formula_for_epsilon}
    v_1\cdots v_\ell=\varepsilon.
\end{equation}
In order to prove this, note that $w=[\overline{a_1;\ldots,a_\ell}]$ is the attracting fixed point of $A=T^{a_1}\cdots V^{a_\ell}$. If we put $M:=\left(\begin{smallmatrix}
    1 & -w \\ 1 & -\tilde{w}
\end{smallmatrix}\right)$ then $MAM^{-1}$ fixes $0$ and $\infty$ and has $0$ as attracting fixed point. It follows that
\(MAM^{-1}= \pm \left(\begin{smallmatrix}
    \varepsilon^{-1} & 0 \\ 0 & \varepsilon
\end{smallmatrix}\right)\)
and 
\((MAM^{-1})'(0)= \varepsilon^{-2},\)
or equivalently $M'(w)A'(w)(M^{-1})'(0)=A'(w)=\varepsilon^{-2}$. This implies that $(A^{-1})'(w)=\varepsilon^{2}$. Letting $\tilde{S}:=\left(\begin{smallmatrix}
    0 & 1 \\ 1 & 0
\end{smallmatrix}\right)$ we can write \(V^aT^b=(\tilde{S}T^a)(\tilde{S}T^b)\) for all $a,b\in \Z$ to get
$$A^{-1}=V^{-a_\ell}T^{-a_{\ell-1}}\cdots V^{-a_2}T^{-a_1}=(\tilde{S}T^{-a_{\ell}})\cdots (\tilde{S}T^{-a_2})(\tilde{S}T^{-a_1}).$$
Using that $\tilde{S}T^{-a}(z)=\frac{1}{z-a}$ and $(\tilde{S}T^{-a})'(z)=-\frac{1}{(z-a)^2}$ we deduce from $(A^{-1})'(w)=\varepsilon^{2}$, and the chain rule for derivatives, the equality
\((-v_1^2)(-v_\ell)^2\cdots (-v_3)^2(-v_2)^2=\varepsilon^{2}.\)
Since $\ell$ is even, we get
\((v_1v_2\cdots v_\ell)^2=\varepsilon^{2}.\)
But $v_1,\ldots,v_\ell>0$ and $\varepsilon>0$, thus 
\(v_1v_2\cdots v_\ell=\varepsilon.\)
This proves \eqref{product_formula_for_epsilon}. Now, using 
\(2\log(\varepsilon)=I_1(A)=S_L(w)+S_L(w^{\op})\) and $S_L(w)=\log(\varepsilon)$, we obtain $S_L(w^\op)=\log(\varepsilon)$. This proves item $(i)$. Finally, by \eqref{eq:real_part_cycle_integral} and the definition of $\hat{S}(w,t)$ in \eqref{eq:def_hat_S} we have
\[\Re(\val_f(w))=\int_{\pi/3}^{2\pi/3}f(e^{it})\hat{S}(w,t).\]
Then item $(ii)$ follows from a duplication and a simple change of variables. This completes the proof of the proposition.
\end{proof}

\begin{example}
For the golden ratio we have $\phi^\op=\phi=[\overline{1;1}]$ hence
$\hat{S}(\phi,t)=\frac{\sin(t)F(\phi,t)}{L(\phi)}$. See Figure \ref{fig:plot_S_golden_ratio} for a  comparison between the plots of $t\mapsto \hat{S}(\phi,t)$ and $t\mapsto \frac{1}{2}\big(\hat{S}(\phi,t)+\hat{S}(\phi,\pi-t)\big)$ for $t\in [\pi/3,2\pi/3]$.

\begin{figure}[h!]
    \centering
    \includegraphics[width=0.6\linewidth]{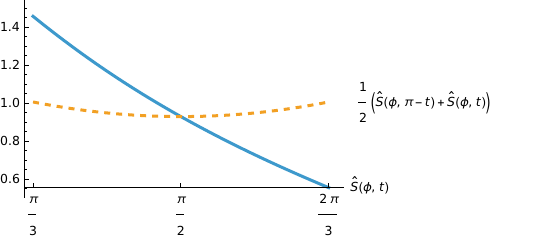}
    \caption{Plots of $\hat{S}(\phi,t)$ and $\frac{1}{2}\big(\hat{S}(\phi,t)+\hat{S}(\phi,\pi-t)\big)$ for $t\in [\pi/3,2\pi/3]$.}
    \label{fig:plot_S_golden_ratio}
\end{figure}
\end{example}

\section{Strategy of proof for Theorem \ref{thm:first_main}}\label{sec:strategy_first_main}

Let $w:=[\overline{a_1;a_2,\ldots,a_\ell}]$ with $\ell\geq 2$ even. Using Proposition \ref{prop:formula_re_valf}$(ii)$, in order to prove $\Re(\val_f(w))\geq \val_f(\phi)$ we need to show that
\begin{equation}\label{eq: goal ineq golden ratio is min}
\int_{\pi/3}^{2\pi/3}f(e^{it})\big(\hat{S}(w,t)+\hat{S}(w,\pi-t)\big)dt\geq \int_{\pi/3}^{2\pi/3}f(e^{it})\big(\hat{S}(\phi,t)+\hat{S}(\phi,\pi-t)\big)dt.
\end{equation}
Let
\begin{equation}\label{eq:def_D_phi_w}
 D_{\phi,w}(t):=\big(\hat{S}(\phi,t)+\hat{S}(\phi,\pi-t)\big)-\big(\hat{S}(w,t)+\hat{S}(w,\pi-t)\big).   
\end{equation}
Clearly $D_{\phi,w}(\pi-t)=D_{\phi,w}(t)$. Also, note that by definition of $\hat{S}(w,t)$ we have
\(\int_{\pi/3}^{2\pi/3}\hat{S}(w,t)dt=1\) for all \(w\). This implies 
\(\int_{\pi/3}^{2\pi/3}D_{\phi,w}(t)dt=0.\)
Finally, observe that \eqref{eq: goal ineq golden ratio is min} is equivalent to
\begin{equation}\label{eq: goal ineq golden ratio is min v2}
   \int_{\pi/3}^{2\pi/3}f(e^{it})D_{\phi,w}(t)dt\leq 0. 
\end{equation}
We will prove below (Sections \ref{sec:first_step_first_main_thm} to \ref{sec:fourth_step_first_main_thm}) that $D_{\phi,w}(t)$ is monotonic decreasing for $t\in [\pi/3,\pi/2]$. Then, the inequality \eqref{eq: goal ineq golden ratio is min v2}, and hence \eqref{eq: goal ineq golden ratio is min} and Theorem \ref{thm:first_main}, are consequence of the following simple lemma by choosing $g(t)=f(e^{it})$ and $D(t)=D_{\phi,w}(t)$.

\begin{lemma}\label{lem: D integral}
    Let $D:[\pi/3,2\pi/3]\to \mathbb{R}$ be a continuous function that satisfies $D(\pi-t)=D(t)$, is decreasing for $t\in[\pi/3,\pi/2]$ and has $\int_{\pi/3}^{2\pi/3}D(t)dt=0$. Let $g:[\pi/3,2\pi/3]\to \mathbb{R}$ be a non-negative continuous function that satisfies $g(\pi -t)=g(t)$ and is increasing for  $t\in [\pi/3,\pi/2]$.  Then
    \(\int_{\pi/3}^{2\pi/3}g(t)D(t)dt\leq 0.\)
\end{lemma}

The proof of Lemma \ref{lem: D integral} is given in the appendix. In conclusion, the proof of Theorem \ref{thm:first_main} reduces to showing that $D_{\phi,w}(t)$ defined in \eqref{eq:def_D_phi_w} is decreasing for $t
\in [\pi/3,\pi/2]$. 

\begin{example}
    When $w=\frac{1+\sqrt{3}}{2}=[\overline{1;2}]$ a plot of $D_{\phi,w}(t)$ is shown in Figure \ref{fig:plot_D_phi,w}.

\begin{figure}[h!]
    \centering
    \includegraphics[width=0.5\linewidth]{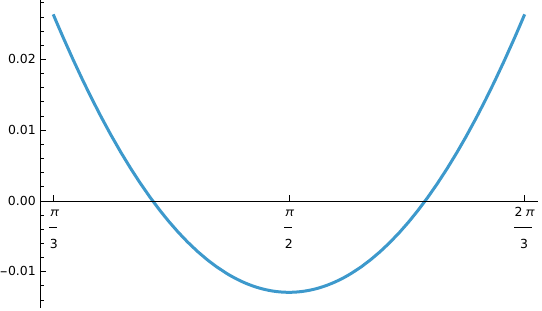}
    \caption{Plot of $D_{\phi,w}(t)$ for $t\in [\pi/3,2\pi/3]$ when $w=\frac{1+\sqrt{3}}{2}=[\overline{1;2}]$.}
    \label{fig:plot_D_phi,w}
\end{figure}
\end{example}

We divide the proof of the fact that $D_{\phi,w}(t)$ is decreasing for $t
\in [\pi/3,\pi/2]$ into four steps.

\subsection{First step: a reduction}\label{sec:first_step_first_main_thm}

We have to check that $\partial_tD_{\phi,w}(t)\geq 0$ for $t\in [\pi/3,\pi/2]$. The computation of the derivative $\partial_tD_{\phi,w}(t)$ leads to the functions
\begin{eqnarray*}
   G(x,t)&:=&\partial_t\big(\sin(t)F(x,t)\big)=\frac{x ((x^2 + 1) \cos(t) - 2 x))}{(x^2 + 1-2 x \cos(t))^2}, \\
   H(x,t)&:=&\partial_t\big(\sin(\pi-t)F(x,\pi-t)\big)=\frac{x ((x^2 + 1) \cos(t) + 2 x))}{(x^2 + 1+2 x \cos(t))^2}.
\end{eqnarray*}
Note that $H(x,t)=-G(x,\pi-t)$. For $w=[\overline{a_1;a_2,\ldots,a_\ell}]$ we then let
\begin{eqnarray*}
   S_G(w,t)&:=&\sum_{i=1}^{\ell}\sum_{m=1}^{a_i}G(w_{i,m},t), \\
S_H(w,t)&:=&\sum_{i=1}^{\ell}\sum_{m=1}^{a_i}H(w_{i,m},t),
\end{eqnarray*}
where $w_{i,m}=[m;\overline{a_{i+1},\ldots, a_\ell,a_1,\ldots,a_{i}}]$ as usual.

For later use we define here for two quadratic irrationals $w_1,w_2\in \R$ with purely periodic continued fraction expansions and $t\in \R$ the function
\begin{equation*}
 D_{w_1,w_2}(t):=\big(\hat{S}(w_1,t)+\hat{S}(w_1,\pi-t)\big)-\big(\hat{S}(w_2,t)+\hat{S}(w_2,\pi-t)\big).   
\end{equation*}

\begin{lemma}\label{lem:D_w1,w2_decreasing_criterion}
Let $w_1,w_2$ be two real quadratic irrationalities with purely periodic continued fraction expansion. Assume that for all $t\in [\pi/3,\pi/2)$ the two inequalities
\begin{eqnarray}\label{eq: two ineq 1}
S_L(w_2)[S_G(w_1,t)+S_H(w_1,t)] &\leq& S_L(w_1)[S_G(w_2,t)+S_H(w_2,t)], \\
S_L(w_2)[S_G(w^\op_1,t)+S_H(w^\op_1,t)] &\leq& S_L(w_1)[S_G(w^\op_2,t)+S_H(w^\op_2,t)], \label{eq: two ineq 2}
\end{eqnarray}
hold. Then, $D_{w_1,w_2}(t)$ is decreasing for $t\in [\pi/3,\pi/2]$. 
\end{lemma}
\begin{proof}
By definition of $D_{w_1,w_2}(t)$ and $\hat{S}(w,t)$ we have 
\begin{eqnarray*}
 \partial_tD_{w_1,w_2}(t)&=&  \left(\frac{S_G(w_1,t)+S_G(w_1^\op,t)+S_H(w_1,t)+S_H(w_1^\op,t)}{2S_L(w_1)}\right)\\
 & & \qquad -\left(\frac{S_G(w_2,t)+S_G(w_2^\op,t)+S_H(w_2,t)+S_H(w_2^\op,t)}{2S_L(w_2)}\right).
\end{eqnarray*}
For $t\in [\pi/3,\pi/2]$ it follows that $\partial_tD_{w_1,w_2}(t)\leq 0$ is equivalent to
\begin{align*}
 & S_L(w_2)\Big(S_G(w_1,t)+S_G(w_1^\op,t)+S_H(w_1,t)+S_H(w_1^\op,t)\Big)\\  
 & \qquad \qquad \qquad \qquad  \leq   S_L(w_1)\Big(S_G(w_2,t)+S_G(w_2^\op,t)+S_H(w_2,t)+S_H(w_2^\op,t)\Big).
\end{align*}
For $t=\pi/2$ we have $G(x,t)+H(x,t)=0$ hence this inequality holds trivially, hence we can assume $t\in [\pi/3,\pi/2)$. Then the desired inequality is a consequence of \eqref{eq: two ineq 1} together with \eqref{eq: two ineq 2}. This proves the lemma.
\end{proof}

\subsection{Second step: good and bad terms}\label{sec:second_step_first_main_thm}

For $w_1=\phi$ and $w_2=w$ the inequality \eqref{eq: two ineq 1} becomes
\begin{eqnarray*}
S_L(w)\Big(S_G(\phi,t)+S_H(\phi,t)\Big) &\leq &S_L(\phi)\Big(S_G(w,t)+S_H(w,t)\Big).
\end{eqnarray*}
Using that the terms $\phi_{i,m}$ defined in \eqref{eq:def_S_F} are just $\phi_{1,1}=\phi$ and $\phi_{2,1}=\phi-1=\frac{1}{\phi}$ and the functions $G(x,t),H(x,t),L(x)$ are invariant under $x\mapsto 1/x$, this inequality becomes
\begin{eqnarray*}
S_L(w)\Big(G(\phi,t)+H(\phi,t)\Big) &\leq &L(\phi)\Big(S_G(w,t)+S_H(w,t)\Big).
\end{eqnarray*}
The above inequality can be rewritten as
\begin{equation}\label{eq:goal_first_thm}
\Big(G(\phi,t)+H(\phi,t)\Big)\sum_{i=1}^{\ell}\sum_{m=1}^{a_i}L(w_{i,m}) \leq L(\phi)\sum_{i=1}^{\ell}\sum_{m=1}^{a_i}\Big(G(w_{i,m},t)+H(w_{i,m},t)\Big).
\end{equation}
Similarly, \eqref{eq: two ineq 2} becomes
\[\Big(G(\phi,t)+H(\phi,t)\Big)\sum_{i=1}^{\ell}\sum_{m=1}^{a_i}L((w^\op)_{i,m}) \leq L(\phi)\sum_{i=1}^{\ell}\sum_{m=1}^{a_i}\Big(G((w^\op)_{i,m},t)+H((w^\op)_{i,m},t)\Big).\]
Since $w^{\op}$ is just another quadratic irrationality with purely periodic continued fraction expansion, it is enough to check \eqref{eq:goal_first_thm} for all possible $w$.

We now introduce the following definition: we say that a term $w_{i,m}$ is \emph{good} if
$$\Big(G(\phi,t)+H(\phi,t)\Big)L(w_{i,m})\leq L(\phi)\Big(G(w_{i,m},t)+H(w_{i,m},t)\Big) \quad \text{for all }t\in [\pi/3,\pi/2),$$
and we say it is \emph{bad} otherwise.

\subsection{Third step: the function $Z(x,t)$}\label{sec:third_step_first_main_thm}

The definition of bad and good terms naturally leads to the function
\begin{equation}\label{eq:def_Z(x,t)}
    Z(x,t):= L(\phi)\big(G(x,t)+H(x,t)\big)-\big(G(\phi,t)+H(\phi,t)\big)L(x).
\end{equation}
Indeed, a term $w_{i,m}$ associated to $w$ is good if $Z(w_{i,m},t)\geq 0$ for all $t\in[\pi/3,\pi/2)$ and it is bad otherwise. Note that \eqref{eq:goal_first_thm} is equivalent to
\begin{equation}\label{eq:goal2_first_thm}
\sum_{i=1}^{\ell}\sum_{m=1}^{a_i}Z(w_{i,m},t)\geq 0.
\end{equation}
Figure \ref{fig:plot_Z} shows the plots of $x\mapsto Z(x,t)$ for $t\in \{\pi/3,5\pi/12,\pi/2\}$. Observe that $Z(x,\pi/2)\equiv 0$.

\begin{figure}[h!]
    \centering
    \includegraphics[width=0.6\linewidth]{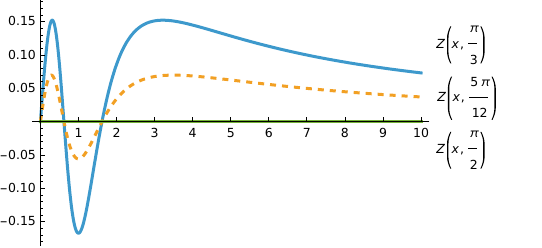}
    \caption{Plots of $x\mapsto Z(x,t)$ for $t\in \{\pi/3,5\pi/12,\pi/2\}$.}
    \label{fig:plot_Z}
\end{figure}

We have the following properties of $Z(x,t)$ whose proofs are postponed to the appendix.

\begin{lemma}\label{lem:first_prop_Z(x,t)}
    For every $t\in [\pi/3,\pi/2)$ the following properties hold:
    \begin{enumerate}
        \item[$(i)$] $Z(1,t)<0$, $Z(\phi,t)=0$ and $Z(x,t)\to 0$ as $x\to \infty$. 
        \item[$(ii)$] There exists a unique $x_t\in (3,4)$ such that the function $x\mapsto Z(x,t)$ is strictly increasing in $[1,x_t]$ and strictly decreasing in $[x_t,\infty)$. 
         \item[$(iii)$] $Z(1,t)+Z(2,t)+\min\{Z(3,t),Z(4,t)\}\geq 0$.
        \item[$(iv)$] $Z(x,t)+Z\left(\frac{1}{x-1},t\right)\geq 0$ for all $x\in \left[\frac{4}{3},\phi\right]$.
    \end{enumerate}
\end{lemma}

The following corollary characterizes the good and bad terms associated to $w$.

\begin{coro}\label{cor:good_and_bad}
    Let $w=[\overline{a_1;a_2,\ldots,a_\ell}]$ be a real quadratic irrationality with purely periodic continued fraction expansion and $w_{i,m}$ as in \eqref{eq:def_S_F}. Then $w_{i,m}$ is a good term if and only if $w_{i,m}\geq \phi$. 
\end{coro}
\begin{proof}
Let $t\in [\pi/3,\pi/2)$. Since $1<\phi= 1.61\ldots <3 <x_t$, by items $(i)$ and $(ii)$ in Lemma \ref{lem:first_prop_Z(x,t)} we have $Z(x,t)<0$ for $x\in [1,\phi)$ and $Z(x,t)\geq 0$ for $x\in [\phi,\infty)$. This implies the result. 
\end{proof}

\subsection{Fourth step: completing the proof}\label{sec:fourth_step_first_main_thm}

We need to prove \eqref{eq:goal2_first_thm} for all $t\in [\pi/3,\pi/2)$. If $w_{i,m}=[m;\overline{a_{i+1},\ldots, a_\ell,a_1,\ldots,a_{i}}]$ is a bad term then $w_{i,m}< \phi$ by Corollary \ref{cor:good_and_bad} and in particular $m=1$. If we further assume that $a_{i+1}\leq 2$ then
$$w_{i,1}=[1;\overline{a_{i+1},\ldots, a_\ell,a_1,\ldots,a_{i}}]\geq 1+\frac{1}{2+\frac{1}{1}}=\frac{4}{3}.$$
Thus $w_{i,1}\in [4/3,\phi]$ and we have $Z(w_{i,1},t)+Z\left(\frac{1}{w_{i,1}-1},t\right)\geq 0$ by Lemma \ref{lem:first_prop_Z(x,t)}$(iv)$. But
$$\frac{1}{w_{i,1}-1}=[\overline{a_{i+1};a_{i+2},\ldots, a_\ell,a_1,\ldots,a_{i}}]=[a_{i+1};\overline{a_{i+2},\ldots, a_\ell,a_1,\ldots,a_{i},a_{i+1}}]> \phi.$$
Hence, we can match the bad term $w_{i,1}$ with  the good term $\frac{1}{w_{i,1}-1}=w_{i+1,a_{i+1}}$ to get a non-negative contribution in \eqref{eq:goal2_first_thm}. 

Now, assume $a_{i+1}\geq 3$. Then
$$w_{i+1,2}=[2;\overline{a_{i+2},\ldots, a_\ell,a_1,\ldots,a_{i},a_{i+1}}]\in (2,3)$$
and
$$w_{i+1,3}=[3;\overline{a_{i+2},\ldots, a_\ell,a_1,\ldots,a_{i},a_{i+1}}]\in (3,4)$$
are good terms. Moreover, by items $(ii)$ and $(iii)$ in Lemma \ref{lem:first_prop_Z(x,t)} we have
$$Z(w_{i,1},t)+Z(w_{i+1,2},t)+Z(w_{i+1,3},t)\geq Z(1,t)+Z(2,t)+\min \{Z(3,t),Z(4,t)\}\geq 0.$$
This shows that we can match the bad term $w_{i,1}$ with $w_{i+1,2}$ together with $w_{i+1,3}$ and get a non-negative contribution in \eqref{eq:goal2_first_thm}. Since all the remaining terms are good we conclude that \eqref{eq:goal2_first_thm} holds. This proves the result.

\section{Strategy of proof for Theorem \ref{thm:second_main}}\label{sec:strategy_second_main}

Let $w$ be a Markov irrationality. Clearly, we can assume $w\neq \psi$. Moreover, by Theorem \ref{thm:first_main} we can assume $w\neq \phi$. We present our proof of Theorem \ref{thm:second_main} in four steps.

\subsection{First step: a reduction}
In order to prove Theorem \ref{thm:second_main} we need to show that
$$\int_{\pi/3}^{2\pi/3}f(e^{it})\Big(\hat{S}(w,t)+\hat{S}(w,\pi-t)\Big)dt\leq \int_{\pi/3}^{2\pi/3}f(e^{it})\Big(\hat{S}(\psi,t)+\hat{S}(\psi,\pi-t)\Big)dt.$$
By Lemma $\ref{lem: D integral}$ it is enough to prove that
$$D_{w,\psi}(t)= \Big(\hat{S}(w,t)+\hat{S}(w,\pi-t)\Big)-\Big(\hat{S}(\psi,t)+\hat{S}(\psi,\pi-t)\Big)$$
is decreasing for $t\in[\pi/3,\pi/2]$. 

Using Lemma \ref{lem:D_w1,w2_decreasing_criterion}, in order to prove that $D_{w,\psi}(t)$ is decreasing for $t\in[\pi/3,\pi/2]$ it is enough to prove the inequalities
\begin{eqnarray}
S_L(\psi)[S_G(w,t)+S_H(w,t)] &\leq& S_L(w)[S_G(\psi,t)+S_H(\psi,t)], \label{eq:first_silver} \\
S_L(\psi)[S_G(w^\op,t)+S_H(w^\op,t)] &\leq& S_L(w)[S_G(\psi,t)+S_H(\psi,t)], \label{eq:second_silver}
\end{eqnarray}
for all Markov irrationalities $w$ and all $t\in [\pi/3,\pi/2]$. Clearly, it is enough to prove \eqref{eq:first_silver} for all quadratic irrationalities of the form $w=[\overline{a_1;a_2,\ldots,a_\ell}]$ with $\ell$ even and $a_i\in \{1,2\}$ for all $i$. 

\subsection{Second step: the function $U(x,t)$}

Let
\begin{equation}\label{eq:def_U}
    U(x,t):= L(x)\Big(S_G(\psi,t)+S_H(\psi,t)\Big)-S_L(\psi)\Big(G(x,t)+H(x,t)\Big). 
\end{equation}
Inequality \eqref{eq:first_silver} is equivalent to
\begin{equation*}
    S_U(w,t):=\sum_{i=1}^{\ell}\sum_{m=1}^{a_i}U(w_{i,m},t)\geq 0.
\end{equation*}
The main properties of the function $U(x,t)$ defined in \eqref{eq:def_U} are the content of the following lemma whose proof is given in the appendix.

\begin{lemma}\label{lem:key_prop_U}
    For every $t\in [\pi/3,\pi/2)$ and every $x\in [\phi,\psi]$ we have
    \begin{enumerate}
        \item[$(i)$] $U(x,t)+U(\Phi(x),t)\geq 0$ and
        \item[$(ii)$] $U(x,t)+U(\Phi(x),t)+U(\Psi(x),t)+U(\Phi\circ\Psi(x),t)\geq 0$.
    \end{enumerate}
\end{lemma}
    
\subsection{Third step: a convenient rearrangement}

Recall that $\Phi=\left(\begin{smallmatrix}
    1 & 1 \\ 1 & 0
\end{smallmatrix}\right)$ and $\Psi=\left(\begin{smallmatrix}
    2 & 1 \\ 1 & 0
\end{smallmatrix}\right)$, which are the generators of the stabilizers in $\GLN$ of $\phi$ and $\psi$, respectively. A simple rearrangement of terms as in \cite[Proof of Proposition 5.1]{BHI25} shows that $S_U(w,t)=S^{(1)}_U(w,t)+S^{(2)}_U(w,t)$ where $S^{(1)}_U(w,t)$ is
\begin{eqnarray*}
&& \sum_{\substack{1\leq i\leq \ell\\ a_{i-1}=2,i \text{ odd}}}\bigg(U([\overline{a_i;a_{i+1},\ldots, a_\ell,a_1,\ldots,a_{i-1}}],t)+U(\Psi([\overline{a_i;a_{i+1},\ldots, a_\ell,a_1,\ldots,a_{i-1}}]),t)\\
& & \quad +U(\Phi([\overline{a_i;a_{i+1},\ldots, a_\ell,a_1,\ldots,a_{i-1}}]),t)+U(\Phi\circ \Psi([\overline{a_i;a_{i+1},\ldots, a_\ell,a_1,\ldots,a_{i-1}}]),t) \bigg),
\end{eqnarray*}
and $S^{(2)}_U(w,t)$ is
\begin{eqnarray*}
    \sum_{\substack{1\leq i\leq \ell\\ a_{i-1}=1,i \text{ odd}}}\bigg( U([\overline{a_i;a_{i+1},\ldots, a_\ell,a_1,\ldots,a_{i-1}}],t)+U(\Phi([\overline{a_i;a_{i+1},\ldots, a_\ell,a_1,\ldots,a_{i-1}}]),t)\bigg).
\end{eqnarray*}

\subsection{Fourth step: completing the proof}

To complete the proof of Theorem \ref{thm:second_main} we also need the following well-known result which can be found, e.g., in \cite[Lemma 1.24]{Aig2013}.

\begin{lemma}\label{lem:comparison_continued_fractions}
    Let $u=[a_1;a_2,\ldots]$, $v=[b_1;b_2,\ldots]$ be two different real numbers. Then $u<v$ if and only if $(-1)^{i+1}a_i<(-1)^{i+1}b_i$ where $i\geq 1$ is the first index for which $a_i\neq b_i$.
\end{lemma}

Recall that $w=[\overline{a_1;a_2,\ldots,a_\ell}]$ with $w\neq \phi$ and $w\neq \psi$. Let $x=[\overline{a_i;a_{i+1},\ldots, a_\ell,a_1,\ldots,a_{i-1}}]$ with $i$ odd. Then, the first appearance of a 1, from left to right, among the partial quotients of $x$ occurs in an odd position, and the same holds for the first appearance of a 2. Hence, by Lemma \ref{lem:comparison_continued_fractions}, we have 
\begin{equation*}
\phi=[\overline{1;1}]< x< [\overline{2;2}]=\psi.
\end{equation*}
In particular, $x\in [1,\psi]$ hence by Lemma \ref{lem:key_prop_U}$(ii)$ we get
\begin{eqnarray*}
 & &  \bigg(U([\overline{a_i;a_{i+1},\ldots, a_\ell,a_1,\ldots,a_{i-1}}],t)+U(\Psi([\overline{a_i;a_{i+1},\ldots, a_\ell,a_1,\ldots,a_{i-1}}]),t)\\
& & \quad +U(\Phi([\overline{a_i;a_{i+1},\ldots, a_\ell,a_1,\ldots,a_{i-1}}]),t)+U(\Phi\circ \Psi([\overline{a_i;a_{i+1},\ldots, a_\ell,a_1,\ldots,a_{i-1}}]),t) \bigg) \geq 0. 
\end{eqnarray*}
This holds, in particular, whenever $i$ is odd and $a_{i-1}=2$, thus $S^{(1)}_U(w,t)$ is a sum of non-negative terms and we conclude that $S^{(1)}_U(w,t)\geq 0$. Similarly, by Lemma \ref{lem:key_prop_U}$(i)$  we have
$$U([\overline{a_i;a_{i+1},\ldots, a_\ell,a_1,\ldots,a_{i-1}}],t)+U(\Phi([\overline{a_i;a_{i+1},\ldots, a_\ell,a_1,\ldots,a_{i-1}}]),t)\geq 0$$
whenever $i$ is odd and $a_{i-1}=1$, thus $S^{(2)}_U(w,t)$ is a sum of non-negative terms and we conclude that $S^{(2)}_U(w,t)\geq 0$. We then obtain $S(w,t)=S^{(1)}_U(w,t)+S^{(2)}_U(w,t)\geq 0$, as desired. This completes the proof of Theorem \ref{thm:second_main}.

\section*{Appendix: Proofs of technical lemmas}

In this appendix, we prove Lemmas \ref{lem: D integral}, \ref{lem:first_prop_Z(x,t)} and \ref{lem:key_prop_U}. The proofs of Lemmas \ref{lem:first_prop_Z(x,t)} and \ref{lem:key_prop_U} are reduced to the verification of numerical bounds for several auxiliary functions, which in turn follow from standard calculations. In most cases we present only a sketch of the calculations for the sake of brevity. 

\subsection*{A.1.~Proof of Lemma \ref{lem: D integral}}

Since $g(t)$ and $D(t)$ are invariant under $t\mapsto \pi-t$, it is enough to prove that
\[\int_{\pi/3}^{\pi/2}g(t)D(t)dt\leq 0.\]
    Let $t_0$ be a zero of $D(t)$ in $[\pi/3,\pi/2]$. It exists since $D(t)$ is continuous and $\int_{\pi/3}^{\pi/2}D(u)du=0$. Put \(I^+=[\pi/3,t_0]\) and \(I^-=[t_0,\pi/2]\). Note that $D(t)\geq 0$  for $t\in I^+$ and $D(t)\leq 0$  for $t\in I^-$ since $D(t)$ is decreasing for $t\in [\pi/3,\pi/2]$. Since $g$ is increasing in $I^+$ and $D(t)\geq 0$ for $t\in I^+$ we have 
    \(\int_{I^+}g(t)D(t)\leq g(t_0)\int_{I^+}D(t)dt\). 
    Now, since $\int_{\pi/3}^{\pi/2}D(t)dt=0$ and $D(t)\leq 0$ for $t\in I^-$ we also have 
    \(\int_{I^+}D(t)= \int_{I^-}|D(t)|dt.\)
    Using that $g$ is also increasing in $I^-$ we conclude that 
    \[\int_{I^+}g(t)D(t)\leq g(t_0)\int_{I^-}|D(t)|dt\leq \int_{I^-}g(t)|D(t)|dt=-\int_{I^-}g(t)D(t)dt, \]
    hence
    \[\int_{\pi/3}^{\pi/2}g(t)D(t)dt=\int_{I^+}g(t)D(t)dt+\int_{I^-}g(t)D(t)dt\leq 0.\]
    This proves the result.

\subsection*{A.2.~Proof of Lemma \ref{lem:first_prop_Z(x,t)}$(i)$}

We have
\[Z(1,t)= L(\phi)\left(\frac{1}{2(\cos(t)-1)}+\frac{1}{2(\cos(t)+1)}\right)-\frac{1}{2}\log(3)\big(G(\phi,t)+H(\phi,t)\big).\]
Letting
\begin{equation}\label{eq:def_P}
P(x,t):=\frac{G(x,t)+H(x,t)}{\cos(t)}=\frac{2 x \left(x^2+1\right) \left(2 x^2 \cos (2 t)+x^4-4 x^2+1\right)}{\left(-2
   x^2 \cos (2 t)+x^4+1\right)^2}    
\end{equation}
we see that $Z(1,t)<0$ is equivalent to 
\begin{equation}\label{eq:Z(1)<0,step2}
    \frac{2L(\phi)}{\cos(t)^2-1}<\log(3)P(\phi,t).
\end{equation}
The left-hand side of \eqref{eq:Z(1)<0,step2} is maximized at $t=\pi/2$ with maximal value $-0.962\ldots$, while the right-hand side of \eqref{eq:Z(1)<0,step2} 
is minimized at $t=\pi/3$ with minimal value $-0.614\ldots$; see Figure \ref{fig:plot_P} for an illustration. This shows that \eqref{eq:Z(1)<0,step2} holds for all $t\in [\pi/3,\pi/2)$, hence $Z(1,t)<0$ as desired.

\begin{figure}[h!]
    \centering
    \includegraphics[width=0.6\linewidth]{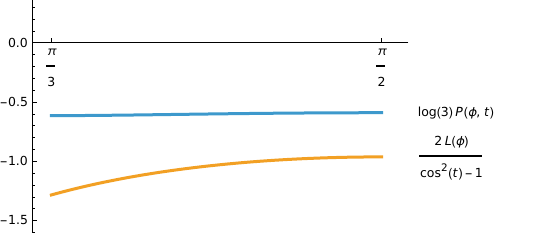}
    \caption{Plots of $\frac{2L(\phi)}{\cos(t)^2-1}$ and $\log(3)P(\phi,t)$ for $t\in [\pi/3,\pi/2]$.}
    \label{fig:plot_P}
\end{figure}
   
The fact that $Z(\phi,t)=0$ follows from the definition of $Z(x,t)$ in \eqref{eq:def_Z(x,t)}. Finally, the fact that  $Z(x,t)\to 0$ as $x\to \infty$ follows from the limits
\[\lim_{x\to \infty}G(x,t)= 0, \quad \lim_{x\to \infty}H(x,t)= 0, \quad \lim_{x\to \infty}L(x)= 0,\]
which in turn follow directly from the definition of each of these functions.

\subsection*{A.3.~Proof of Lemma \ref{lem:first_prop_Z(x,t)}$(ii)$}

By Lemma \ref{lem:first_prop_Z(x,t)}$(i)$ it is enough to show that for fixed $t$ the partial derivative $\partial_x Z(x,t)$ vanishes at exactly one point $x_t\in (3,4)$ for $x\in (1,\infty)$. We start by computing
\begin{eqnarray*}
    \partial_x G(x,t)
    &=& -\frac{(x^2 - 1) ((x^2 + 1) \cos(t) + 2 x \cos^2(t) - 4 x)}{(-2 x \cos(t) + x^2 + 1)^3},\\
    \partial_x H(x,t) 
    &=&  -\frac{(x^2 - 1) ((x^2 + 1) \cos(t) - 2 x \cos^2(t) + 4 x)}{(2 x \cos(t) + x^2 + 1)^3}, \\
    L'(x)&=&\frac{1 - x^2}{x^4 + x^2 + 1}.
\end{eqnarray*}
For $x>1$ we get that $\partial_x Z(x,t)=0$ if and only if
    \begin{eqnarray*}
 & & L(\phi)\left( \frac{ ((x^2 + 1) \cos(t) + 2 x \cos^2(t) - 4 x)}{(-2 x \cos(t) + x^2 + 1)^3} + \frac{ ((x^2 + 1) \cos(t) - 2 x \cos^2(t) + 4 x)}{(2 x \cos(t) + x^2 + 1)^3}\right)\\
 &= & \left( G(\phi,t)+H(\phi,t)\right)\frac{1}{x^4 + x^2 + 1}.
    \end{eqnarray*}
Let~$R(x,t)$ be the function
\begin{eqnarray*}
& &\frac{1}{\cos(t)}\left(\frac{ ((x^2 + 1) \cos(t) + 2 x \cos^2(t) - 4 x)}{(-2 x \cos(t) + x^2 + 1)^3} + \frac{ ((x^2 + 1) \cos(t) - 2 x \cos^2(t) + 4 x)}{(2 x \cos(t) + x^2 + 1)^3}\right)\\
&=&\frac{2 \left(12 x^6 \cos (2 t)+16 x^4 \cos (2 t)+2 x^4 \cos (4 t)+12
   x^2 \cos (2 t)+x^8-8 x^6-28 x^4-8 x^2+1\right)}{\left(-2 x^2 \cos
   (2 t)+x^4+1\right)^3}.
\end{eqnarray*}
Then, for $x>1$ we have~$\partial_x Z(x,t)=0$ if and only if
    \begin{eqnarray}\label{eq:LRP}
 L(\phi)R(x,t)= \frac{P(\phi,t)}{x^4 + x^2 + 1},
    \end{eqnarray}
    where $P(x,t)$ is defined in \eqref{eq:def_P}. 
    This equality is invariant under $x\mapsto 1/x$, hence we can instead check that it has only one solution $\tilde{x}_t:=x_t^{-1}\in (0,1)$ with $\tilde{x}_t\in(1/4,1/3)$. We first check that \eqref{eq:LRP} has a solution $\tilde{x}_t\in (1/4,1/3)$. For $x=1/4$ the function $t\mapsto L(\phi)R(1/4,t)-\frac{P(\phi,t)}{(1/4)^4 + (1/4)^2 + 1}$ for $t\in [\pi/3,\pi/2]$ is minimized at $t=\pi/2$ with minimal value $0.2229\ldots$, hence
\[L(\phi)R(1/4,t)- \frac{P(\phi,t)}{(1/4)^4 + (1/4)^2 + 1}\geq 0.222\ldots.\]
For $x=1/3$ the function $t\mapsto L(\phi)R(1/3,t)-\frac{P(\phi,t)}{(1/3)^4 + (1/3)^2 + 1}$ for $t\in [\pi/3,\pi/2]$ is maximized at $t=\pi/3$ with maximal value $-0.202\ldots$, hence
\[L(\phi)R(1/3,t)-\frac{P(\phi,t)}{(1/3)^4 + (1/3)^2 + 1}\leq -0.202\ldots.\]
See Figure \ref{fig:Plot_exists_xt} for a comparison of the plots of these functions.

\begin{figure}[h!]
    \centering 
    \includegraphics[width=0.6\linewidth]{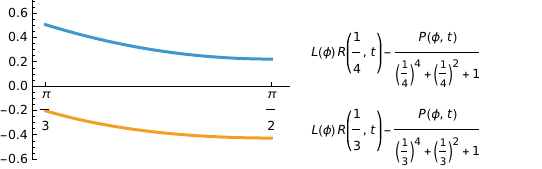}
    \caption{Plots of $L(\phi)R(x,t)- \frac{P(\phi,t)}{x^4 + x^2 + 1}$ for $x=1/3$ and $x=1/4$ as functions of $t\in [\pi/3,\pi/2]$.}
    \label{fig:Plot_exists_xt}
\end{figure}

We conclude that equation \eqref{eq:LRP} has at least one solution $\tilde{x}_t\in (1/4,1/3)$. We will now show that this is the only solution in the domain $x\in (0,1)$. Since $P(\phi,t)<0$ (see Figure \ref{fig:plot_P}) and $x\mapsto \frac{1}{x^4 + x^2 + 1}$ 
is decreasing for $x\in [0,1]$, we have that $\frac{P(\phi,t)}{x^4 + x^2 + 1}$ is increasing for $x\in [0,1]$. Then, the uniqueness of the solution $\tilde{x}_t$ of \eqref{eq:LRP} in the domain $x\in [0,1]$ follows from the following properties of the function $L(\phi)R(x,t)$ valid for every fixed $t\in [\pi/3,\pi/2)$:
\begin{enumerate}
    \item[$(a)$] $L(\phi)R(x,t)$ is decreasing for $x\in [0,1/2]$.
    \item[$(b)$] $L(\phi)R(x,t)< \frac{P(\phi,t)}{x^4+x^2+1}$ for all $x\in [1/2,1]$.
\end{enumerate}
Indeed, property $(a)$ together with the fact that $\frac{P(\phi,t)}{x^4 + x^2 + 1}$ is increasing for $x\in [0,1]$ ensures that the solution $\tilde{x}_t$ of \eqref{eq:LRP} is unique in the domain $x\in [0,1/2]$, while property $(b)$ ensures that \eqref{eq:LRP} has no solution for $x\in [1/2,1]$.

In order to prove $(a)$ we start by noting that $L(\phi)=\log(\phi) >0$ and writing $R(x,t)=2M(x^2,2t)$ where
\[M(x,t):=\frac{ \left(12 x^3 \cos (t)+16 x^2 \cos (t)+2 x^2 \cos (2 t)+12
   x \cos ( t)+x^4-8 x^3-28 x^2-8 x+1\right)}{\left(-2 x \cos
   ( t)+x^2+1\right)^3}\]
so that $\partial_xR(x,t)=4\,\partial_xM(x^2,2t)x$. A computation shows that 
\(\partial_x M(x,t)=-\frac{2N(x,t)}{\left(-2 x \cos (t)+x^2+1\right)^4}\)
where
\begin{eqnarray*}
  N(x,t) &:=& 19 x^4 \cos ( t)+32 x^3 \cos ( t)+4 x^3 \cos (2 t)+39
   x^2 \cos ( t)-8 x^2 \cos (2 t)-x^2 \cos (3 t)\\ & & -14 x \cos (2 t)-9
   \cos ( t)+x^5-12 x^4-58 x^3-16 x^2+19 x+4.
\end{eqnarray*}
We have to check that $N(x,t)\geq 0$ for all $x\in [0,1/4]$ and all $t\in [2\pi/3,\pi]$. Standard calculations show that $N(x,t)$ is minimized at $(x,t)=(0,2\pi/3)$ with minimal value $8.5$; see Figure \ref{fig:Plot_N} for a plot. This shows that $N(x,t)\geq 0$ holds for all $(x,t)\in [0,1/4]\times [2\pi/3,\pi]$ and $(a)$ is proven.

\begin{figure}[h!]
    \centering 
    \includegraphics[width=0.4\linewidth]{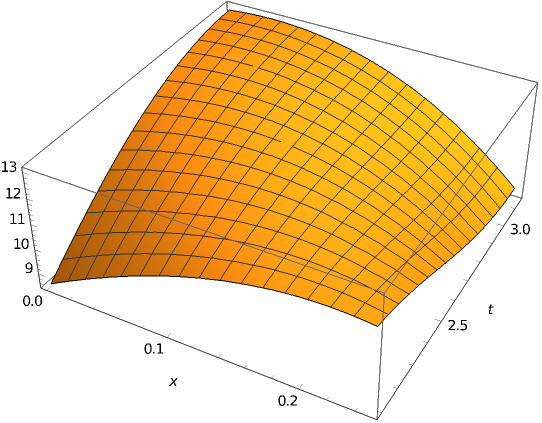}
    \caption{Plot of $N(x,t)$ for $(x,t)\in [0,1/4]\times [2\pi/3,\pi]$. The point $(0,2\pi/3)$ is a global minimum with value $8.5$.}
    \label{fig:Plot_N}
\end{figure}

Finally, item $(b)$ follows from the facts that for $(x,t)\in[1/2,1]\times [\pi/3,\pi/2]$ the function $L(\phi)R(x,t)$ is maximized at $(x,t)=(1,\pi/2)$ with maximal value $-1.203\ldots$ while  $\frac{P(\phi,t)}{x^4+x^2+1}$ is minimized at $(x,t)=(1/2,\pi/3)$ with minimal value $-0.425\ldots$; see Figure \ref{fig:plots_LR_P} for the plots of these two functions. This proves $(b)$ and completes the proof of Lemma \ref{lem:first_prop_Z(x,t)}$(ii)$. 

\begin{figure}[h!]
    \centering
    \subfloat{\includegraphics[width=0.4\linewidth]{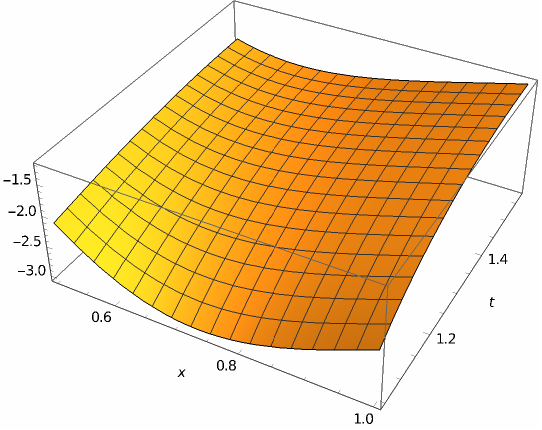}} \quad \quad 
    \subfloat{\includegraphics[width=0.4\textwidth]{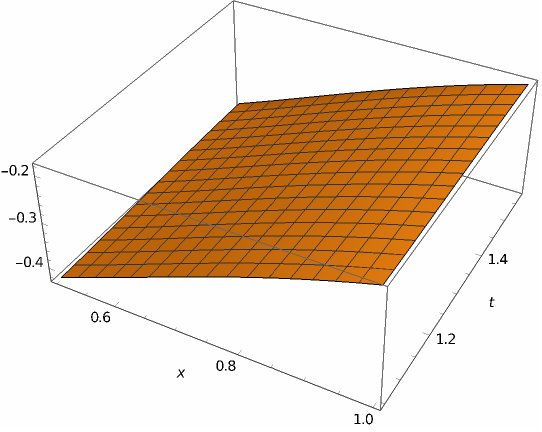}} 
\caption{Plots of $L(\phi)R(x,t)$ on the left and of $\frac{P(\phi,t)}{x^4+x^2+1}$ on the right for $(x,t)\in [1/2,1]\times [\pi/3,\pi/2]$.} \label{fig:plots_LR_P}
\end{figure}

\subsection*{A.4.~Proof of Lemma \ref{lem:first_prop_Z(x,t)}$(iii)$}

Let $\tilde{Z}(x,t):=\frac{Z(x,t)}{\cos(t)}$. We have to prove that \(\tilde{Z}(1,t)+\tilde{Z}(2,t)+\min\{\tilde{Z}(3,t),\tilde{Z}(4,t)\}> 0.\) 
We have
\begin{eqnarray*}
\tilde{Z}(1,t) &=& \frac{\log(\phi)}{\cos(t)^2-1}-\log(3)\frac{\sqrt{5}(4\cos(t)^2-3)}{(4\cos(t)^2-5)^2}, \\ 
\tilde{Z}(2,t) &=& \log(\phi) \frac{(320\cos(t)^2 - 140)}{(16\cos(t)^2 - 25)^2}-\log\left(\frac{7}{3}\right)\frac{\sqrt{5}(4\cos(t)^2-3)}{(4\cos(t)^2-5)^2}, \\
\tilde{Z}(3,t) &=&  \log(\phi) \frac{(135\cos(t)^2 + 105)}{(9\cos(t)^2 - 25)^2}- \log\left(\frac{13}{7}\right)\frac{\sqrt{5}(4\cos(t)^2-3)}{(4\cos(t)^2-5)^2}, \\\
\tilde{Z}(4,t) &=&  \log(\phi) \frac{(8704\cos(t)^2 + 21896)}{(64\cos(t)^2 - 289)^2} - \log\left(\frac{21}{13}\right)\frac{\sqrt{5}(4\cos(t)^2-3)}{(4\cos(t)^2-5)^2}.
\end{eqnarray*}
A calculation shows that $\tilde{Z}(1,t)\geq \tilde{Z}(1,\pi/3)=-0.334\ldots$, $\tilde{Z}(2,t)\geq \tilde{Z}(2,\pi/2)=0.119\ldots$,  $\tilde{Z}(3,t)\geq \tilde{Z}(3,\pi/2)=0.246\ldots$ and $\tilde{Z}(4,t)\geq \tilde{Z}(4,\pi/2)=0.254\ldots$ for all $t\in[\pi/3,\pi/2]$. 
In particular, $\tilde{Z}(1,t)+\tilde{Z}(2,t)+\min\{\tilde{Z}(3,t),\tilde{Z}(4,t)\}\geq 0.031\ldots >0$. This proves the result.


\subsection*{A.5.~Proof of Lemma \ref{lem:first_prop_Z(x,t)}$(iv)$}

We have to prove that
\begin{eqnarray*}
    L(\phi)(G(x,t)+H(x,t))-(G(\phi,t)+H(\phi,t))L(x)&\\
    +L(\phi)\left(G\left(\frac{1}{x-1},t\right)+H\left(\frac{1}{x-1},t\right)\right)-(G(\phi,t)+H(\phi,t))L\left(\frac{1}{x-1}\right)&\geq 0
\end{eqnarray*}
for all $(x,t)\in [4/3,\phi]\times [\pi/3,\pi/2)$. 
Using the function $P(x,t)$ defined in \eqref{eq:def_P}, this is equivalent to
\begin{eqnarray*}
    L(\phi)\left(P(x,t)+P\left(\frac{1}{x-1},t\right)\right) \geq  P(\phi,t)\left(L(x)+L\left(\frac{1}{x-1}\right)\right), 
\end{eqnarray*}
hence also equivalent to
\begin{eqnarray*}
    L(\phi)P(x,t)-P(\phi,t)L(x) \geq  P(\phi,t)L\left(\frac{1}{x-1}\right)-L(\phi)P\left(\frac{1}{x-1},t\right).
\end{eqnarray*}
This inequality follows directly from the following two facts:
\begin{enumerate}
    \item[$(a)$] The function $\frac{L(\phi)P(x,t)-P(\phi,t)L(x)}{\phi-x}$ for $(x,t)\in [4/3,\phi)\times [\pi/3,\pi/2]$ is minimized at $(x,t)=(4/3,\pi/3)$ with minimal value \(-0.651\ldots\).
    \item[$(b)$] The function $\frac{P(\phi,t)L\left(\frac{1}{x-1}\right)-L(\phi)P\left(\frac{1}{x-1},t\right)}{\phi-x}$ for $(x,t)\in [4/3,\phi)\times [\pi/3,\pi/2]$ is maximized at $(x,t)=(4/3,\pi/2)$ with maximal value \(-0.867\ldots\). 
\end{enumerate}
In order to prove $(a)$ we first compute 
\[\partial_t P(x,t)=-\frac{8x^3 \sin(2t) (x^2 + 1) (2x^2 \cos(2t) - 8x^2 + 3x^4 + 3)}{(x^4 - 2 \cos(2t) x^2 + 1)^3}\]
and
\[\partial_x\partial_tP(x,t)=\frac{8 x^2 (x^2-1) \sin(2t)p(x,t)}{(1 + x^4 - 2 x^2 \cos(2 t))^4}\]
where
\(p(x,t):= 9 + 9 x^8 + (4 x^2 + 4 x^6) (-4 + 7 \cos(2 t)) + 2 x^4 (-38 + 16 \cos(2 t) + \cos(4 t)).\)
Then
\begin{equation}\label{(iv)_a1}
L(\phi)\partial_x\partial_t P(x,t)\leq \partial_tP(\phi,t)L'(x)
\end{equation}
is equivalent to
\begin{equation}\label{(iv)_a2}
\frac{\log(\phi)x^2 p(x,t)}{(1 + x^4 - 2 x^2 \cos(2 t))^4}\leq -\frac{\sqrt{5}  (2 \cos(2t) + 1)}{(2 \cos(2t) - 3)^3  (1+x^2 + x^4)}.
\end{equation}
For $(x,t)\in [4/3,\phi]\times [\pi/3,\pi/2]$ the left-hand side of \eqref{(iv)_a2} is maximized at $(x,t)=(\phi,\pi/2)$ with maximal value $-0.019\ldots$, while the right-hand side of \eqref{(iv)_a2} is minimized at $(x,t)=(4/3,\pi/2)$ with minimal value \(-0.003\ldots\); see Figure \ref{fig:Plot3d_two} for the plots of these two functions. 
This implies that \eqref{(iv)_a2} and \eqref{(iv)_a1} hold. 

\begin{figure}[h!]
    \centering 
    \includegraphics[width=0.4\linewidth]{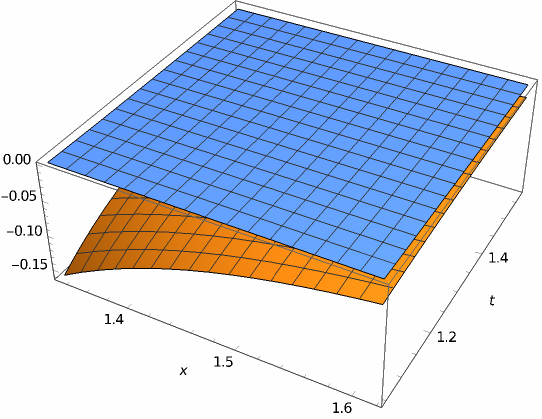}
    \caption{Plots of $\frac{\log(\phi)x^2 p(x,t)}{(1 + x^4 - 2 x^2 \cos(2 t))^4}$ (below) and $ -\frac{\sqrt{5}  (2 \cos(2t) + 1)}{(2 \cos(2t) - 3)^3  (1+x^2 + x^4)}$ (above) for $(x,t)\in [4/3,\phi]\times [\pi/3,\pi/2]$.}
    \label{fig:Plot3d_two}
\end{figure}

Now, \eqref{(iv)_a1} shows that for fixed $t$ the difference $L(\phi)\partial_t P(x,t)- \partial_tP(\phi,t)L(x)$ is minimized when $x=\phi$, so $L(\phi)\partial_t P(x,t)- \partial_tP(\phi,t)L(x)\geq 0$ for all $(x,t)\in [4/3,\phi]\times [\pi/3,\pi/2)$. This, in turn, shows that for fixed $x$ the difference $L(\phi)P(x,t)- P(\phi,t)L(x)$ is minimized when $t=\pi/3$ hence
\begin{align*}
  L(\phi)P(x,t)- P(\phi,t)L(x)&\geq L(\phi)P(x,\pi/3)- P(\phi,\pi/3)L(x) \\
 &  = \frac{\sqrt{5}}{8} \log\left(\frac{x^2 + x + 1}{x^2 - x + 1}\right) + \frac{2x \log\left(\phi\right) (x^2 + 1) (x^4 - 5x^2 + 1)}{(x^4 + x^2 + 1)^2}.  
\end{align*}
Another simple calculation shows that
\begin{equation}\label{eq:function_for_plot}
    \frac{1}{\phi-x}\left(\frac{\sqrt{5}}{8} \log\left(\frac{x^2 + x + 1}{x^2 - x + 1}\right) + \frac{2x \log\left(\phi\right) (x^2 + 1) (x^4 - 5x^2 + 1)}{(x^4 + x^2 + 1)^2}\right)
\end{equation}
for $x\in [4/3,\phi)$ is minimized at $x=4/3$; see Figure \ref{fig:Mat_Plot4} for a plot. Then, evaluating this expression at $x=4/3$ gives $(a)$.

\begin{figure}[h!]
    \centering 
    \includegraphics[width=0.4\linewidth]{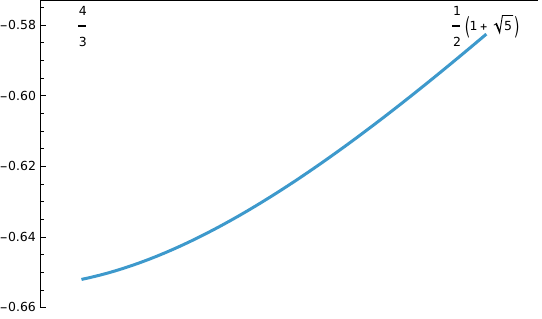}
    \caption{Plot of function \eqref{eq:function_for_plot} for $x\in [4/3,\phi)$.}
    \label{fig:Mat_Plot4}
\end{figure}

In order to prove $(b)$ we start by computing
\begin{align*}
\partial_tP(\phi,t)L\left(\frac{1}{x-1}\right)&= \frac{4 \sqrt{5} \sin (2 t) (2 \cos (2 t)+1) \log \left(\frac{x^2-x+1}{x^2-3   x+3}\right)}{(2 \cos (2 t)-3)^3}, \\
L(\phi)\partial_tP\left(\frac{1}{x-1},t\right)&= -\frac{8 (x-1)^3 \left(x^2-2 x+2\right) \log \left(\phi\right) \sin (2 t) q(x,t)}{\left(\left(x^2-2 x+2\right)^2-4 (x-1)^2 \cos   ^2(t)\right)^3},
\end{align*}
where
\(q(x,t):=2 (x-1)^2 \cos (2 t)+3 x^4-12   x^3+10 x^2+4 x-2.\)
Then 
\begin{equation}\label{eq:A5b}
\partial_tP(\phi,t)L\left(\frac{1}{x-1}\right)\geq L(\phi)\partial_tP\left(\frac{1}{x-1},t\right)
\end{equation}
for $(x,t)\in [4/3,\phi]\times [\pi/3,\pi/2)$ is equivalent to
\begin{equation}\label{eq:A5b2}
 1 \leq \frac{2(3-2 \cos (2 t))^3 \log \left(\phi\right) (x-1)^3 \left(x^2-2 x+2\right) q(x,t)}{ \sqrt{5} (2 \cos (2 t)+1)\left(\left(x^2-2 x+2\right)^2-4 (x-1)^2 \cos   ^2(t)\right)^3 \log \left(\frac{x^2-x+1}{x^2-3   x+3}\right)}.
\end{equation}
A simple computation shows that
\begin{equation}\label{eq:A5b3}
0<\frac{2\log(\phi) (x-1)^3(x^2-2x+2)}{\sqrt{5}\log \left(\frac{x^2-x+1}{x^2-3   x+3}\right)}\leq \frac{1}{\phi^{4}}.
\end{equation}
On the other hand, for fixed $t$ the function $x\mapsto \frac{ q(x,t)}{\left(\left(x^2-2 x+2\right)^2-4 (x-1)^2 \cos   ^2(t)\right)^3}$ is decreasing for $x\in [4/3,\phi]$, which implies that
\[\frac{ q(x,t)}{\left(\left(x^2-2 x+2\right)^2-4 (x-1)^2 \cos   ^2(t)\right)^3}\geq  \phi^4\frac{(2\cos(2t)+1)}{(3-2\cos(2t))^3}.\]
Noting that the righ-hand side is negative, we obtain
\begin{equation}\label{eq:A5b4}
\frac{ (3-2\cos(2t))^3  q(x,t)}{(2\cos(2t)+1)\left(\left(x^2-2 x+2\right)^2-4 (x-1)^2 \cos   ^2(t)\right)^3}\leq  \phi^4.
\end{equation}
Combining \eqref{eq:A5b4} with \eqref{eq:A5b3} we obtain \eqref{eq:A5b2} which in turn shows \eqref{eq:A5b}. Now, \eqref{eq:A5b} implies that for fixed $x$ the function $P(\phi,t)L\left(\frac{1}{x-1}\right)-L(x)P\left(\frac{1}{x-1},t\right)$ is maximized when $t=\pi/2$, hence
\begin{align*}
   & P(\phi,t)L\left(\frac{1}{x-1}\right)-L(x)P\left(\frac{1}{x-1},t\right)\leq P\left(\phi,\frac{\pi}{2}\right)L\left(\frac{1}{x-1}\right)-L(x)P\left(\frac{1}{x-1},\frac{\pi}{2}\right)\\
   & \quad \quad \quad = -\frac{3 \left(47+21 \sqrt{5}\right) \log \left(\frac{x^2-x+1}{x^2-3
   x+3}\right)}{525+235 \sqrt{5}}-\frac{2\left(x^5-5 x^4+4 x^3+8 x^2-12 x+4\right)
   \log \left(\phi\right)}{\left(x^2-2 x+2\right)^3}. 
\end{align*}
A simple calculation shows that
\begin{equation}\label{eq:function_for_plot_2}
    \frac{1}{\phi-x}\left(-\frac{3 \left(47+21 \sqrt{5}\right) \log \left(\frac{x^2-x+1}{x^2-3
   x+3}\right)}{525+235 \sqrt{5}}-\frac{2\left(x^5-5 x^4+4 x^3+8 x^2-12 x+4\right)
   \log \left(\phi\right)}{\left(x^2-2 x+2\right)^3}\right)
\end{equation}
for $x\in [4/3,\phi)$ is maximized at $x=4/3$; see Figure \ref{fig:Mat_Plot5} for a plot. Finally, evaluating this expression at $x=4/3$ gives $(b)$.

\begin{figure}[h!]
    \centering 
    \includegraphics[width=0.4\linewidth]{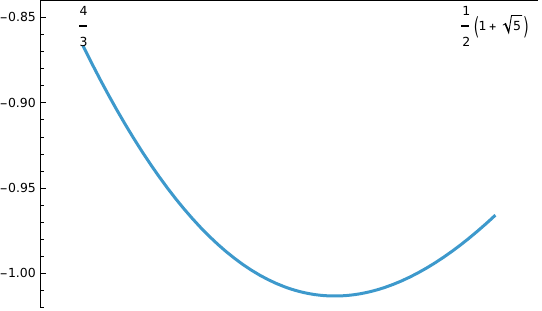}
    \caption{Plot of function \eqref{eq:function_for_plot_2} for $x\in [4/3,\phi)$.}
    \label{fig:Mat_Plot5}
\end{figure}

\subsection*{A.6.~Proof of Lemma \ref{lem:key_prop_U}$(i)$} Let
\(\tilde{U}(x,t):=\frac{U(x,t)}{2\cos(t)}.\) Then
\[\tilde{U}(x,t)=L(x)\Big(P(\psi,t)+P(\psi-1,t)\Big)-\Big(L(\psi)+L(\psi-1)\Big)P(x,t).\]
We have to prove that $\tilde{U}(x,t)+\tilde{U}\left(\Phi(x),t\right)\geq 0$ for all $(x,t)\in [\phi,\psi]\in [\pi/3,\pi/2)$. For $x=\psi$ we have  $\tilde{U}(x,t)+\tilde{U}\left(\Phi(x),t\right)=0$, hence it is enough to prove that for fixed $t$ the function $x\mapsto \tilde{U}(x,t)+\tilde{U}\left(\Phi(x),t\right)$ is decreasing for $x\in [\phi,\psi]$. 
We have
\begin{equation*}
L(x)+L\left(\Phi(x)\right) = \frac{1}{2}\log\left(\frac{3x^2+3x+1}{x^2-x+1}\right), \qquad L(\psi)+L(\psi-1)=\log(\psi)
\end{equation*}
and
\begin{equation*}
    K(t):=P(\psi,t)+P(\psi-1,t)=\frac{7  \sqrt{2}  \left(63\cos(t)^2 - 54\cos(t)^4 + 16\cos(t)^6 - 24\right)}{\left(8\cos(t)^4 - 25\cos(t)^2 + 18\right)^2}.
\end{equation*}
This shows that $\tilde{U}(x,t)+\tilde{U}\left(\Phi(x),t\right)$ equals
\[ \frac{1}{2}\log\left(\frac{3x^2+3x+1}{x^2-x+1}\right)K(t) -\log(\psi)\left(P(x,t)+P\left(1+\frac{1}{x},t\right)\right),\]
hence $\partial_x\left(\tilde{U}(x,t)+\tilde{U}\left(\Phi(x),t\right)\right)\leq 0$ if and only if 
\[\frac{x^2(2 + 2 x - 3 x^2)}{(1 + 2 x + x^2 + 3 x^4)}K(t)\leq \log(\psi)\left(x^2\partial_xP(x,t)-\partial_xP\left(1+\frac{1}{x},t\right)\right).\]
Then the desired result follows from the following uniform bounds valid for all $(x,t)\in [\phi,\psi]\times [\pi/3,\pi/2]$:
\begin{enumerate}
    \item[$(a)$] $\frac{x^2(2 + 2 x - 3 x^2)}{(1 + 2 x + x^2 + 3 x^4)}K(t)\leq 0.5$,
    \item[$(b)$] $\log(\psi)x^2\partial_xP(x,t)\geq 2.1$,
    \item[$(c)$] $\log(\psi)\partial_xP\left(1+\frac{1}{x},t\right)\leq 1.4$.
\end{enumerate}
Item $(a)$ follows from the bounds 
\[-0.546\ldots \leq \frac{x^2(2 + 2 x - 3 x^2)}{(1 + 2 x + x^2 + 3 x^4)}\leq -0.25, \qquad -0.750\ldots \leq K(t)\leq -0.733\ldots.\]
For item $(b)$ we compute
\[\partial_xP(x,t)=\frac{2 \left(1-x^2\right) \left(1 + x^8 + \left(4 x^2 + 4 x^6\right) \left(-2 + 3 \cos(2 t)\right) + 2 x^4 \left(-14 + 8 \cos(2 t) + \cos(4 t)\right)\right)}{\left(1 + x^4 - 2 x^2 \cos(2 t)\right)^3}.\]
Then $(b)$ is equivalent to
\begin{align*}
& \log(\psi) x^2\left(1-x^2\right)\left(1 + x^8 + \left(4 x^2 + 4 x^6\right) \left(-2 + 3 \cos(2 t)\right) + 2 x^4 \left(-14 + 8 \cos(2 t) + \cos(4 t)\right)\right)\\
& \geq  1.05\left(1 + x^4 - 2 x^2 \cos(2 t)\right)^3.
\end{align*}
After expanding, collecting terms and putting $y:=x^2$ the above inequality turns out to be equivalent to
\begin{align*}
-\frac{1.05}{y}  \geq  & \, \, -\log   (\psi )-6.3 \cos (2 t) +y \left(-12 \log (\psi ) \cos   (2 t)+9 \log (\psi )+12.6 \cos ^2(2 t)+3.15\right)\\
& +y^2 \left(-4 \log (\psi ) \cos (2 t)-2 \log (\psi ) \cos (4 t)+20   \log (\psi )-8.4 \cos ^3(2 t)-12.6 \cos (2 t)\right) \\
& +y^3 \left(4 \log (\psi  ) \cos (2 t)+2 \log (\psi ) \cos (4 t)-20 \log (\psi )+12.6 \cos ^2(2   t)+3.15\right)\\
& + y^4 (12 \log (\psi ) \cos (2 t)-9 \log (\psi )-6.3 \cos (2 t)) +y^5 (\log (\psi )+1.05)
\end{align*}
for $(y,u)\in [\phi^2,\psi^2]\times [\pi/3,\pi/2]$. The left-hand side is minimized at $y=\phi^2$ with minimal value $-0.401\ldots$, while the right-hand side is maximized at $(y,t)=(\phi^2,\pi/2)$ with maximal value $-29.385\ldots$. This proves $(b)$. Finally, item $(c)$ is equivalent to
\begin{align*}
& \log(\psi)\left(1-z^2\right)\left(1 + z^8 + \left(4 z^2 + 4 z^6\right) \left(-2 + 3 \cos(2 t)\right) + 2 z^4 \left(-14 + 8 \cos(2 t) + \cos(4 t)\right)\right)\\
& \leq 0.7\left(1 + z^4 - 2 z^2 \cos(2 t)\right)^3
\end{align*}
for $z:=1+\frac{1}{x}\in [\sqrt{2},\phi]$ and $t\in [\pi/3,\pi/2]$. Letting $w:=z^2$ this turns out to be equivalent to
\begin{align*}
    \frac{\log(\psi)-0.7}{w} \leq & \, \,  -12 \log (\psi ) \cos (2 t)+9 \log (\psi )-4.2 \cos (2 t)
     \\
    & +w \left(-4 \log  (\psi ) \cos (2 t)-2 \log (\psi ) \cos (4 t)+20 \log (\psi )+8.4 \cos ^2(2   t)+2.1\right)
    \\
 &   +w^2 \left(4 \log (\psi ) \cos (2 t)+2 \log (\psi )   \cos (4 t)-20 \log (\psi )-5.6 \cos ^3(2 t)-8.4 \cos (2 t)\right)\\
& +w^3 \left(12 \log (\psi ) \cos (2 t)-9 \log (\psi   )+8.4 \cos ^2(2 t)+2.1\right) \\
& +w^4 (\log (\psi )-4.2 \cos (2 t)) +0.7 w^5
\end{align*}
for $(w,t)\in [2,\phi^2]\times [\pi/3,\pi/2]$. The left-hand side is maximized at $w=2$ with  maximal value $0.090\ldots$, while the right-hand side is minimized at $(w,t)=(2,\pi/3)$ with  minimal value $0.714\ldots$. This proves $(c)$ and completes our proof of Lemma \ref{lem:key_prop_U}$(i)$.

\subsection*{A.7.~Proof of Lemma \ref{lem:key_prop_U}$(ii)$} Using the function \(\tilde{U}(x,t)=\frac{U(x,t)}{2\cos(t)}\) as in the proof of  Lemma \ref{lem:key_prop_U}$(i)$ we have to prove that $V(x,t):=\tilde{U}(x,t)+\tilde{U}\left(\Phi(x),t\right)+\tilde{U}\left(\Psi(x),t\right)+\tilde{U}\left(\Phi\circ \Psi(x),t\right)\geq 0$. Using that
\begin{eqnarray*}
L(x)+L(\Phi(x))+L(\Psi(x))+L(\Phi\circ \Psi(x))=\frac{1}{2}\log\left(\frac{19x^2+15x+3}{x^2-x+1}\right)
\end{eqnarray*}
we get that $V(x,t)$ equals
\[\frac{1}{2}\log\left(\frac{19x^2+15x+3}{x^2-x+1}\right)K(t)-\log(\psi)\left(P(x,t)+P\left(\Phi(x),t\right)+P\left(\Psi(x),t\right)+P\left(\Phi\circ \Psi(x),t\right)\right).\]
Since this function vanishes at $x=\psi$, in order to prove that it is non-negative for $(x,t)\in [\phi,\psi]\times [\pi/3,\pi/2)$ it is enough to prove that $\partial_xV(x,t)\leq 0$. This is equivalent to
\begin{align*}
    & \frac{(9 + 16x - 17 x^2)x^2}{(1 -x+x^2) (3 + 15x + 19 x^2)}K(t) \\
    &\qquad \leq  \log(\psi)\left(x^2\partial_xP(x,t)-\partial_xP(\Phi(x),t)-\partial_xP(\Psi(x),t)+\frac{x^2}{(2x+1)^2}\partial_xP(\Phi\circ \Psi(x),t)\right).
\end{align*}
Then the desired result follows from the following uniform bounds valid for all $(x,t)\in [\phi,\psi]\times [\pi/3,\pi/2]$:
\begin{enumerate}
    \item[$(a)$] $\frac{(9 + 16x - 17 x^2)x^2}{(1 -x+x^2) (3 + 15x + 19 x^2)}K(t)\leq 0.34$,
    \item[$(b)$] $\log(\psi)x^2\partial_xP(x,t)\geq 2.1$,
    \item[$(c)$] $\log(\psi)\partial_xP\left(\Phi(x),t\right)\leq 1.4$.
    \item[$(d)$] $\log(\psi)\partial_xP\left(\Psi(x),t\right)\leq 0.38$.
    \item[$(e)$] $\log(\psi)\frac{x^2}{(2x+1)^2}\partial_xP\left(\Phi\circ \Psi(x),t\right)\geq 0.12$.
\end{enumerate}
Items $(b)$ and $(c)$ were already verified in the proof of Lemma \ref{lem:key_prop_U}$(i)$. Items $(a)$, $(d)$ and $(e)$ can be verified by similar methods,  hence we omit the details. 

\end{document}